\documentclass{article}

\usepackage[text={5.7in, 7.9in},centering]{geometry}
\usepackage{amsmath,amsthm,xcolor,cite}
\usepackage{bm,amssymb,mathrsfs,url,units,dsfont}
\usepackage{setspace}%
\usepackage{bbold}
\DeclareMathOperator\Cov{Cov}
\DeclareMathOperator\Corr{Corr}

\DeclareMathOperator\Dimh{Dim_{_{\rm H}}}

\renewcommand{\P}{\mathds{P}}
\newcommand{\E}{\mathds{E}}
\newcommand{\R}{\mathds{R}}
\newcommand{\Z}{\mathds{Z}}

\newcommand{\cS}{\mathcal{S}}
\newcommand{\cV}{\mathcal{V}}
\newcommand{\cP}{\mathcal{P}}

\renewcommand{\d}{\mathrm{d}}

\newcommand{\cab}{\bm{c}_z}
\newcommand{\e}{\mathrm{e}}

\newcommand{\ul}{u^{(\ell)}}

\newtheorem{stat}{Statement}[section]
\newtheorem{proposition}[stat]{Proposition}

\newtheorem{theorem}[stat]{Theorem}
\newtheorem{lemma}[stat]{Lemma}
\theoremstyle{definition}
\newtheorem{definition}[stat]{Definition}
\newtheorem{remark}[stat]{Remark}

\numberwithin{equation}{section}

\begin{document}%\onehalfspacing

\title{On the large-scale structure of the tall peaks for stochastic heat equations with fractional Laplacian %
	\thanks{Research supported by the NSF Grant No. 0932078000 while the author was in residence at the Mathematical Sciences Research Institute in Berkeley, California, during the Fall 2015 semester.}
}
\author{
	Kunwoo Kim
}

\date{\today}

\maketitle

\begin{abstract}
We consider stochastic heat equations with fractional Laplacian on $\R^d$.  Here,  the driving noise is generalized Gaussian  which is white in time but spatially homogenous and the spatial covariance is given by the Riesz kernels. We study the large-scale structure of the tall peaks for (i) the linear stochastic heat equation and (ii) the parabolic Anderson model.   We obtain the largest order of the tall peaks  and  compute the macroscopic Hausdorff dimensions of the tall peaks for both (i) and (ii). These results imply that both (i) and (ii) exhibit multi-fractal behavior in a macroscopic scale even though (i) is not intermittent and  (ii) is intermittent. This is an extension of  a recent result  of \cite{KKX} to a wider class of stochastic heat equations.\\
 
\noindent{\it Keywords:} Stochastic heat equations, fractional Laplacian, Riesz kernels,  macroscopic Hausdorff dimension\\
	
	\noindent{\it \noindent AMS 2010 subject classification:}
	Primary. 60H15; Secondary. 35R60, 60K37.
\end{abstract}

\section{Introduction and main results}
In \cite{KKX}, Khoshenvisan et al  consider, among many other things, the following stochastic heat equation
\begin{equation}\label{eq:SHE}
 \frac{\partial }{\partial t} u_t(x) = \frac{\partial^2}{\partial x^2} u_t(x) + \sigma(u_t(x)) \eta, \quad t>0, \, x\in \R,
\end{equation}subject to the initial function $u_0(x)$. Here, $\eta$ is space-time white noise, i.e. its covariance is given by
\begin{equation*}
\Cov(\eta_t(x),\eta_s(y))=\delta(|t-s|) \delta(|x-y|).
\end{equation*}
In particular, they consider two cases: (i) the \emph{linear stochastic heat equation}, i.e., $\sigma(u)=1$ in \eqref{eq:SHE} (we denote the solution to the linear stochastic heat equation as $Z_t(x)$), and (ii) the \emph{parabolic Anderson model}, i.e., $\sigma(u)=u$ in \eqref{eq:SHE} (we denote the solution to the parabolic Anderson model as $u_t(x)$).

 One of the interesting properties for the parabolic Anderson model is that the solution $u_t(x)$ is \emph{intermittent}, i.e.,   $\E|u_t(x)|^k \approx \exp\{\gamma(k)t\}$ as $t\rightarrow \infty$ where $\gamma(k)/k>0$ is strictly increasing for $k\geq 1$ (see \cite{CM} and \cite{ZRS}). Intuitively, intermittency means that the solution develops many different sizes of high peaks on small regions of different scales. Thus, \emph{intermittency} and \emph{multi-fractality} are often regarded as the same property. On the other hand, the linear stochastic heat equation is a Gaussian random field, which implies that it is \emph{not} intermittent. However, one very interesting result from \cite{KKX} is  that even the linear stochastic heat equation which is not intermittent can also be \emph{multi-fractal} in a \emph{macroscopic} scale as  the parabolic Anderson model is. More precisely, they consider the following sets: 
\begin{equation}\label{sets}
\begin{aligned}
& \cP_{Z_t}(\gamma):=\left\{ x\in\R^d: \, \|x\|>\e, \, Z_t(x)\geq  (t/\pi)^{1/4}\sqrt{2\gamma \log x} \right\},\\
 &\cP_{u_t}(\gamma):=\left\{ x\in\R^d: \, \|x\| >\e, \, \log u_t(x)\geq \gamma t^{1/3}(\log x)^{2/3} \right\}.
 \end{aligned}
 \end{equation} Note that $\log u_t(x)$ is the so-called Hopf-Cole solution to the KPZ equation of the statistical mechanics (\cite{KPZ}). Here, the sets in \eqref{sets} can be considered  as the sets of the tall peaks for the linear stochastic heat equation and the KPZ equation where the largest orders of the tall peaks for fixed $t>0$ are given by the functions $\sqrt{\log \|x\|}$ and $(\log x)^{2/3}$ respectively as $x\to \infty$ (those functions are called \emph{gauge} functions). Here, $\gamma$ can be regarded as a \emph{scale} parameter which scales the heights of the tall peaks. Khoshnevisan et al in \cite{KKX} show that  the {macroscopic} Hausdorff dimensions (introduced by Barlow and Taylor in \cite{BarlowTaylor1, BarlowTaylor} and denoted as $\Dimh$) of the sets in \eqref{sets} are 
 \[ \Dimh\left(\mathcal{P}_{Z_t}(\gamma) \right)=1-\gamma \quad \text{and} \quad \Dimh\left(\cP_{u_t}(\gamma) \right)= 1-\frac{4\sqrt{2}}{3}\gamma^{3/2}.\]
 This means that for infinitely many different $\gamma$'s, we have all different macroscopic Hausdorff dimensions, which says that both the linear stochastic heat equation and the parabolic Anderson model are \emph{multi-fractal} (see \cite[Definition 1.1]{KKX}). 
 
  The main objective of this paper is to extend the result from \cite{KKX} mentioned in the above paragraph to a wider class of stochastic heat equations. In other words, we consider the following linear stochastic heat equation with fractional Laplacian 
  
 \begin{equation}\label{linearFSHE}
\left[ \begin{split}
 \frac{\partial }{\partial t} Z_t(x) &= -(-\Delta)^{\alpha/2} Z_t(x) + \dot{F}_t(x), \quad t>0, \, x\in \R^d,\\
 Z_0(x)&=0, \quad x\in\R^d,
\end{split}\right.
\end{equation}
and the  parabolic Anderson model for fractional Laplacian 
\begin{equation}\label{FSHE}
 \left[\begin{split}
 \frac{\partial }{\partial t} u_t(x) &= -(-\Delta)^{\alpha/2} u_t(x) + u_t(x)\dot{F}_t(x), \quad t>0, \, x\in \R^d,\\
 u_0(x)&=1, \quad x\in\R^d.
\end{split}\right.
\end{equation}
Here, $-(-\Delta)^{\alpha/2}$ for $\alpha \in (0,2]$ is the so-called  fractional Laplacian which is the infinitesimal generator of a symmetric $\alpha$-stable process $\{X_t\}_{t\geq 0}$ in $\R^d$ whose L\'evy exponent is given by $\psi(\xi)=\|\xi\|^{\alpha}$, i.e., 
\[\E\left[\e^{i\xi\cdot X_t}\right]=\e^{-t \psi(\xi)}.\] Note that when $\alpha=2$,  $-(-\Delta)^{\alpha/2}$ is just Laplacian $\Delta$ which is the  infinitesimal generator of Brownian motion in $\R^d$.
 
 In addition, in \eqref{linearFSHE} and \eqref{FSHE},  $\dot{F}$ is  a spatially homogeneous generalized Gaussian random field with covariance of the form 
\begin{equation}
\Cov(\dot F(t,x),\dot F(s,y))=\delta(|t-s|) f(|x-y|),
\end{equation} where  $f$ is of  Riesz type kernel, i.e., 
 \begin{equation}\label{eq:Riesz}
f(z):=c_{\beta,d}\|z\|^{-\beta} \quad\text{and}\quad   c_{\beta,d}:= \frac{2^\beta\, \pi^{d/2}\, \Gamma(\beta/2)}{\Gamma((d-\beta)/2)},
\end{equation}  for $\beta \in (0, d)$. 
In this way, we have 
\begin{equation}
\hat f (\xi)=\|\xi\|^{-d+\beta}
\end{equation}
where $\hat f$ is the Fourier transform of $f$, i.e., $\hat f(\xi):=\int_{\R^d}  f(x) \exp(-i\, \xi\cdot x)\, \d x$.  Here, one can  see that as $\beta\to d$, $\dot F$ becomes space-time white noise. On the other hand, as $\beta \to 0$, $\dot F$ becomes white noise in time only, i.e., $F$ is just a one-dimensional Brownian motion $B_t$ which is independent of the space variable $x$. In this case ($\beta=0$), the solutions to \eqref{linearFSHE} and \eqref{FSHE} are $Z_t(x)=B_t$ and $\log u_t(x)=B_t -t/2$, which do not have any chaotic spatial structure.

Partial differential equations (both stochastic and deterministic PDEs) with fractional Laplacian    have received much attention last decades. In addition, the existence and uniqueness of solutions of stochastic heat equations driven by spatially colored noise  is also well-studied. In particular, it is  well-known that both \eqref{linearFSHE} and \eqref{FSHE} have unique solutions as long as $0<\beta < \alpha\wedge d$ (see, e.g., \cite{FD}). Thus, throughout this paper, we assume that
\begin{equation}\label{condition}
0<\beta<\alpha\wedge d \leq 2\wedge d.
\end{equation}

 We note that when $\alpha \in (0,2)$, $-(-\Delta)^{\alpha/2}$ is a \emph{non-local} operator. In addition, when $\beta\in (0,d)$, the spatial correlation between $\dot F(t,x)$ and $\dot F(t,y)$ does not vanish even if $\|x-y\|$ is very large, which is different from the case where $\beta=d$, i.e., $\dot F$ is space-time white noise. These  non-local property of the fractional Laplacian and long-range correlation of the noise affect to the correlation length and also the largest order of the tall peaks for both the  linear stochastic heat equation with fractional Laplacian \eqref{linearFSHE} and the parabolic Anderson model for fractional Laplacian \eqref{FSHE}. In addition, it is known that the parabolic Anderson model for fractional Laplacian is intermittent (at least weakly) but the linear stochastic heat equation is not since it  is a Gaussian random field (see \cite{FKN, FD}). On the other hand, our main theorems (Theorems \ref{th:linear} and \ref{th:pam} below)  show that the tall peaks for both \eqref{linearFSHE} and \eqref{FSHE} show multi-fractal behavior, i.e., for infinitely different scale paramter $\gamma$'s, we get all different macroscopic Hausdorff dimensions of the tall peaks. 

Let $\log_{+} r:=\log(r\vee \e)$ for $r\in \R$. Here is our first main theorem about the tall peaks for the linear stochastic heat equation with fractional Laplacian:
\begin{theorem}\label{th:linear}
Define 
\begin{equation}\label{eq:set:linear}
 \cP_{Z_t}(\gamma):=\left\{ x\in\R^d: \, Z_t(x)\geq \left(2\, \cab \, t^{(\alpha-\beta)/\alpha}\, \gamma\, \log_+\|x\|  \right)^{1/2} \right\},
 \end{equation} where $\cab:=\frac{c_{\beta,d}\Gamma(\beta/\alpha)}{(\alpha-\beta)2^{\beta/\alpha}}$ and $\Gamma(x)$ is the Gamma function.

Then, for every $t>0$ and $\gamma>0$, we have
\begin{equation}
\Dimh \left[ \cP_{Z_t}(\gamma)\right]=(d-\gamma) \vee  0, \quad a.s..
\end{equation}
In addition, we also have
\begin{equation}\label{eq:sup:linear}
\limsup_{\|x\|\to \infty} \frac{Z_t(x)}{\sqrt{\log\|x\|}} = \sqrt{2d\bm{c}_z t^{(\alpha-\beta)/\alpha}}.
\end{equation}

\end{theorem}

We now consider the parabolic Anderson model for fractional Laplacian. 
\begin{theorem}\label{th:pam}
Define 
\begin{equation}\label{eq:set:pam}
 \cP_{u_t}(\gamma):=\left\{ x\in\R^d: \, \log_{+} u_t(x)\geq \gamma\, t^{(\alpha-\beta)/(2\alpha-\beta)}(\log_+\|x\|)^{\alpha/(2\alpha-\beta)}    \right\}.
 \end{equation}
Then, there exists two constants $0<{\bm{c}} \leq {\bm{C}} <\infty$ which only depend on $\alpha, \beta, d$ such that, for every $t>0$ and $\gamma>0$, we have
\begin{equation}
 0\vee \left(d-{\bm{C}}\gamma^{(2\alpha-\beta)/\alpha}\right)\leq  \Dimh \left[ \cP_{u_t}(\gamma)\right]\leq \left(d-{\bm{c}}\gamma^{(2\alpha-\beta)/\alpha}\right)\vee 0, \quad a.s..
\end{equation}
In addition, we have
\begin{equation}
 t^{(\alpha-\beta)/(2\alpha-\beta)} \left(\frac{d}{\bm{C}}\right)^{\alpha/(2\alpha-\beta)}\leq \limsup_{\|x\|\to \infty} \frac{\log_+ u_t(x) }{(\log \|x\|)^{\alpha/(2\alpha-\beta)}} \leq  t^{(\alpha-\beta)/(2\alpha-\beta)} \left(\frac{d}{\bm{c}}\right)^{\alpha/(2\alpha-\beta)}.
 \end{equation}
\end{theorem}

Khoshnevisan et al in \cite{KKX} provide certain conditions for obtaining the upper and lower bounds of the macroscopic Hausdorff dimensions of the sets of the tall peaks for general random fields. Regarding the condition for  the lower bound,  the main point  is to construct some independent random variables which are close to the original random field.  For the construction of those independent random variables, we use quite different approaches. For Theorem \ref{th:linear}, since $\{Z_t(x)\}$ is a Gaussian random field, we use Berman's theorem (\cite{Berman}) and Slepian's inequality (\cite{Slepian}) to construct independent random variables (see also Remark \ref{rem:lower}). In this way, computations are much simpler than the ones for Theorem \ref{th:pam}. On the other hand, for Theorem \ref{th:pam}, we follow a localization argument developed by Conus et al in \cite{CJKS} (see Section \ref{sec:lower:pam}). 
Regarding the proof of the upper bounds,  we will verify the condition \eqref{cond:UB}. This condition comes from a natural way, i.e., we can get the condition when we use a covering of boxes with side 1 (we explain how \eqref{cond:UB} leads to the upper bound in the proof of the upper bound of Theorem \ref{th:linear}, right after Lemma \ref{lem:linear:tail}). Since $Z_t(x)$ is Gaussian, as for the lower bound, we use Qualls and Watanabe's result (\cite{QW}) on gaussian random fields for the upper bound of Theorem \ref{th:linear}. On the other hand, for Theorem \ref{th:pam}, we apply a quantitative form of Kolmogorov's continuity theorem to verify \eqref{cond:UB}.

This paper is organized as follows: Section \ref{sec:prelim} introduces the macroscopic Hausdorff dimension and certain conditions for obtaining the upper and lower bounds of the Hausdorff dimensions of the tall peaks. In Section \ref{sec:linear}, we  prove Theorem \ref{th:linear}. Section \ref{sec:FSHE} deals with the parabolic Anderson model for fractional Laplacian and provides a proof of Theorem \ref{th:pam}.

\section{Preliminaries}\label{sec:prelim}
\subsection{Macroscopic Hausdorff dimension}
We first introduce the macroscopic Hausdorff dimension defined by Barlow and Taylor  \cite{BarlowTaylor1, BarlowTaylor}.  
For all integers $n\geq 1$, we define
\begin{equation}
	\cV_n :=\left[-\e^n\,, \e^n\right)^d, \quad
	\cS_0 :=\cV_0,  \quad\text{and}\quad
	\cS_{n+1}  :=\cV_{n+1}\setminus\cV_n.
\end{equation}

\begin{definition}
	Let $\mathcal{C}$ denote the collection of all cubes of the form
	\begin{equation}\label{box}
		Q(x\,,r) := \left[ x_1\,,x_1+r\right)\times\cdots\times
		\left[ x_d\,,x_d+r\right),
	\end{equation}
	as $x :=(x_1\,,\ldots,x_d)\in \R^d$ and $r\in [1, \infty)$. For $Q(x,r) \in \mathcal{C}$, we call $r$ the \emph{side} of $Q$, denoted as $s(Q)=r$.  Let $E \subset \R^d$. We now define, 
 for any number $\rho>0$, and all integers $n\geq 1$,
\begin{equation}\label{nu:rho}
	\nu^n_\rho(E) := \min \left\{ \sum_{i=1}^m\left(\frac{s(Q_i)}{\e^n}\right)^\rho\, : Q_i \in \mathcal{C}, Q_i \subset \cS_n,  E\cap \cS_n \subset \cup_{i=1}^m Q_i    \, \right\}.
\end{equation} \end{definition}

\begin{definition}
	The Barlow--Taylor \emph{macroscopic Hausdorff dimension} 
	of $E\subseteq\R^d$ is defined as 
	\begin{equation} 
		\Dimh E :=
		\inf\left\{\rho>0:\ \sum_{n=1}^\infty\nu^n_\rho(E)
		<\infty\right\}=\sup \left\{\rho>0:\ \sum_{n=1}^\infty\nu^n_\rho(E)
		=\infty\right\}.
	\end{equation}
\end{definition}

\subsection{General bounds}\label{sec:general}
Let $X:=\{X(x); x \in \R^d\}$ be a  real-valued random field. For all real numbers $b\in(0\,,\infty)$
we  define
\begin{equation}\label{eq:gen:tail:LB}
	\bm{c}(b) := -\limsup_{z\to\infty} z^{-b}
	\sup_{x\in \R^d}\log\P\left\{ X(x)> z\right\},
\end{equation}
and
\begin{equation}\label{eq:gen:tail:UB}
	\bm{C}(b):=
	-\liminf_{z\to\infty} z^{-b}\inf_{x\in \R^d }\log\P\left\{ X(x)> z\right\}.
\end{equation}
We now define the set of the tall peaks for $X$ as  
\[ \cP_{X}^{(c)} (\gamma) := \left\{x\in \R^d :\ \|x\|>\exp(\e),\
		X(x) \ge \left(\frac{\gamma}{c}\log \|x\|\right)^{1/b}\right\}. \]

Khoshnevisan et al provide  some conditions for upper and lower bounds of the macroscopic Hausdorff dimension of $\cP_X^{(c)}(\gamma)$ in \cite{KKX}. We first consider the condition for the upper bound, which is given in \cite[Theorem 4.1]{KKX}.

\begin{theorem}[A general upper bound]\label{th:upper:general}
Suppose that there exists $b\in(0\,,\infty)$
	such that $\bm{c}(b) \in [0,\infty)$ and for all $\gamma\in(0\,,d)$,
	\begin{equation}\label{cond:UB}
		\sup_{y\in \R^d}\P\left\{\sup_{x\in Q(y,1)}X(x) > 
		\left( \frac{\gamma}{\bm{c}(b)}\,\log s\right)^{1/b}\right\}
		\le s^{-\gamma + o(1)}\text{ as $s\to\infty$}.
	\end{equation}
	%Then,
	%\begin{equation}\label{gen:LIL:UB}
	%	\limsup_{\|x\|\to\infty} \frac{X(x)}{(\log \|x\|)^{1/b}}
	%	\le \left(\frac{d}{\bm{c}(b)}\right)^{1/b} \qquad\text{a.s.}
	%\end{equation}
	Then we have
	\begin{equation}\label{dim:gen:UB}
		\Dimh \left[\cP_X^{(\bm{c}(b))}(\gamma) \right]\le
		d-\gamma, \text{a.s.}
	\end{equation}
	for all $\gamma\in(0\,,d)$.
	In addition, we have
	\[ \limsup_{\|x\|\to \infty} \frac{X(x)}{\log \|x\|} \leq \left(\frac{d}{\bm{c(b)}} \right)^{1/b}.
	\]
\end{theorem}
The proof of Theorem \ref{th:upper:general} is given in \cite[Theorem 4.1]{KKX}. For readers who may wonder how \eqref{cond:UB} is a natural condition,  we give a quick  explanation how \eqref{cond:UB} implies the upper bound in the proof of Theorem \ref{th:linear}.  We will use Theorem \ref{th:upper:general} to get the upper bounds in both of Theorem \ref{th:linear} and Theorem \ref{th:pam}.

 Let us now consider the condition for the lower bound. We first give some notations. Let
\begin{equation}
	\Pi_n(\theta) := A_n(\theta)\times\cdots\times A_n(\theta)
	\qquad[d\text{ times}];
\end{equation}
where
\begin{equation}
	A_n(\theta) := \bigcup_{\substack{%
	0\le j\le \e^{n(1-\theta)} :\\ j\in\Z}}
	\left\{ \e^n + j \e^{\theta n} \right\}
\end{equation}
This $\{\Pi_n(\theta)\}_{n=0}^\infty$ is called a $\theta$-skeleton of $\R^d$ in \cite[Definition 4.2]{KKX}. 

\begin{definition}\label{def:thick}
	Let $E\subseteq\R^d$ be a set and choose and fix
	some real number $\theta\in(0\,,1)$. We say that
	$E$ is \emph{$\theta$-thick} if there exists 
	an integer $M=M(\theta)$ such that 
	\begin{equation}
		E\cap Q(x\,,\e^{\theta n})\neq\varnothing,
	\end{equation}
	for all $x\in\Pi_n(\theta)$ and $n\ge M$.
\end{definition}
 
 When a set $E$ is a \emph{uniform} set, i.e., a set of points uniformly spaced in $\R^d$, then the macroscopic Hausdorff dimension of $E$ is easily computed. For example, when $U=\bigcup_{n=1}^\infty \Pi_n(\theta)$,  we get $\Dimh(U)=d(1-\theta)$ (see \cite[Section 5]{BarlowTaylor}). Therefore, when the set $E\subset \R^d$ is $\theta$-thick or contains a $\theta$-thick  set, we can expect to get the lower bound as follows:
\begin{proposition}[Proposition 4.4, \cite{KKX}]\label{pr:thick}
	If $E\subset\R^d$ is $\theta$-thick or contains a set which is $\theta$-thick for some 
	$\theta\in(0\,,1)$, then 
	\begin{equation}
		\Dimh E\ge d(1-\theta).
	\end{equation}
\end{proposition}
%See Proposition 4.4 in \cite{KKX} for the proof. 

We now give the following general lower bound statement which is Theorem 4.7 in \cite{KKX}.

\begin{theorem}[A general lower bound]\label{th:Gen:LIL:LB}
	Suppose there exists $b\in(0\,,\infty)$ such that $\bm{C}(b)>0$. 
	Suppose in addition that for any $\delta\in (0,1)$  there exists an increasing nonrandom 
	measurable function $S:\R\to\R$ such that as $n\to\infty$
	\begin{equation}\label{cond:LB}
		n^{-1} \max_{\{t_i\}_{i=1}^m\in\Pi_n(\delta)}
		\max_{1\le j\le m}
		\inf_{\{Y_i\}_{i=1}^m\in\mathcal{I}}\log\P \{ |S(X_{t_j})
		- S(Y_j)|> 1\} \to -\infty,
	\end{equation}where $\mathcal{I}$ denotes the collection of all independent
	finite sequences of independent random variables. If $\gamma\in(0\,,d)$, then 
	\begin{equation}\label{eq:Gen:Dim:LB}
		\Dimh \left[ \cP_X^{(\bm{C}(b))} (\gamma) \right]\ge
		d-\gamma\ \, \text{a.s.}
	\end{equation}
	In addition, we have
	\[ \liminf_{\|x\|\to \infty} \frac{X(x)}{\log \|x\|} \geq \left(\frac{d}{\bm{C(b)}} \right)^{1/b}.
	\]
\end{theorem}

\begin{remark}
A rigorous proof of Theorem \ref{th:Gen:LIL:LB} is given in \cite[Theorem 4.7]{KKX}. Here we give a brief explanation how the condition \eqref{cond:LB} provides the lower bound. This explanation suggests a probability estimate that can lead to the lower bound. For simplicity, assume $S(x)=x$.  Let $x\in \Pi_n(\theta)$.  Let $m$ refer to the number of all the elements in $\Pi_n(\delta)\cap Q(x,\e^{n\theta})$, i.e.,  there exists a constant $c>1$ such that $c^{-1}\exp\left(nd(\theta-\delta)\right)\leq m \leq c\exp\left(nd(\theta-\delta)\right)$.  Thanks to \eqref{cond:LB} and \eqref{eq:gen:tail:UB}, for any $\epsilon \in (0, d\theta-d\delta-\gamma)$, we may have 
\begin{equation}
\begin{aligned}
 \P& \left\{\max_{\{y_i\}_{i=1}^m \in \Pi_n(\delta)\cap  Q\left( x,\e^{n\theta}\right)}\frac{X(y_i)}{(\log\|y_i\|)^{1/b}}<\left(\frac{\gamma}{\bm{C}(b)}\right)^{1/b} \right\}\\ \label{eq:lower}
  &\lessapprox \P\left\{ \max_{1\leq i\leq m} Y(i)<\left(\frac{\gamma n}{\bm{C}(b)}\right)^{1/b} \right\}=  \prod_{i=1}^m \P\left\{ Y(i) <\left(\frac{\gamma n}{\bm{C}(b)}\right)^{1/b}\right\}\\
&=  \prod_{i=1}^m \left(1 - \P\left\{ Y(i) >\left(\frac{\gamma n}{\bm{C}(b)}\right)^{1/b}\right\}\right) \lessapprox \prod_{i=1}^m \left(1 - \P\left\{ X(x_i) >\left(\frac{\gamma n}{\bm{C}(b)}\right)^{1/b}\right\}\right)\\
& \leq \prod_{i=1}^m \left(1-\e^{-(\gamma+\epsilon)n}\right)\leq \exp\left(-\e^{n(-\gamma-\epsilon+d\theta-d\delta)} \right).
\end{aligned}
\end{equation} Since  $\epsilon \in (0, d\theta-d\delta-\gamma)$, we obtain
\begin{align*}
&\sum_{n=0}^\infty  \P\left\{\text{There exists $x\in \Pi_n(\theta)$ such that $\max_{y\in \Pi_n(\delta)\cap Q\left( x,\e^{n\theta}\right)}\frac{X(y)}{(\log\|y\|)^{1/b}}<\left(\frac{\gamma}{\bm{C}(b)}\right)^{1/b}$  } \right\}\\
& \leq \sum_{n=1}^\infty \sum_{x \in \Pi_n(\theta)} \P \left\{\max_{\{y_i\}_{i=1}^m \in \Pi_n(\delta)\cap  Q\left( x,\e^{n\theta}\right)}\frac{X(y_i)}{(\log\|y_i\|)^{1/b}}<\left(\frac{\gamma}{\bm{C}(b)}\right)^{1/b} \right\}\\
&\leq \sum_{n=1}^\infty \exp\left(nd(1-\theta)\right) \exp\left(-\e^{n(-\gamma-\epsilon+d\theta-d\delta)} \right) <\infty.
\end{align*}
Hence,  the Borel-Cantelli lemma says that $\left\{x\in \R^d :\ \|x\|>\exp(\e),\
		\frac{X(x)}{(\log \|x\|)^{1/b}} \ge 
		\left(\frac{\gamma}{\bm{C}(b)}\right)^{1/b}\right\}$ contains a $\theta$-thick set almost surely for $\theta>(\epsilon+\gamma+d\delta)/d$. We now use Proposition \ref{pr:thick} and let $\epsilon$ and $\delta$ go to 0 to get  that the dimension is bounded below by $d-\gamma$. 
\end{remark}

\begin{remark}\label{rem:lower}
The condition \eqref{cond:LB} is called a coupling condition since we couple the random field $\{X(x_i)\}$ and $\{Y_i\}$ together. However, we may be able to loose the coupling condition since all we want is a probability estimate in \eqref{eq:lower}. Thus,  for the lower bound in Theorem \ref{th:linear}, we use a \emph{less strict} coupling argument  to show that the set defined in \eqref{eq:set:linear} contains a $\gamma/d$-thick set. More specifically, since $\{Z_t(x)\}_{x\in\R^d}$ is a stationary Gaussian random field, we can construct  Gaussian random variables $\{\tilde Z_i\}$ not from $\{Z_t(x)\}_{x\in\R^d}$, but still can get a similar probability estimate as in \eqref{eq:lower} by using some properties of Gaussian random variables such as Slepian's inequality (\cite{Slepian}) and Berman's theorem for equally correlated random variables (\cite{Berman}). This provides us much simpler computations than the one for Theorem \ref{th:pam} where we construct independent random variables by using a localization technique developed by Conus et al in \cite{CJKS}.
\end{remark}

\section{Linear stochastic heat equation with fractional Laplacian} \label{sec:linear}
Recall
\begin{equation}\label{linearFSHE1}
 \begin{aligned}
 \frac{\partial }{\partial t} Z_t(x) &= -(-\Delta)^{\alpha/2} Z_t(x) + \dot{F}_t(x), \quad t>0, \, x\in \R^d,\\
 Z_0(x)&=0, \quad x\in\R^d.
\end{aligned}
\end{equation}
It is well-known (see, e.g., \cite{FKN}) that a \emph{mild} solution to \eqref{linearFSHE1} satisfies the following integral equation:
\begin{equation}\label{eq:soln:linear}
Z_t(x)=\int_{(0,t)\times \R^d} p_{t-s}(y-x)\, F (\d s, \d y),
\end{equation}
where $p_t(x)$ is the transition density of the symmetric $\alpha$-stable process whose L\'evy exponent is given by $\psi(\xi)=\|\xi\|^{\alpha}$. This also says that  $\hat p_t(\xi)=\e^{-t\psi(\xi)}$. In addition, it is well-known that there exists two constants $0<c_1 <c_2<\infty$ such that
\begin{equation}\label{eq:density}
c_1 \left\{t^{-\alpha/d} \wedge \frac{t}{|x-y|^{d+\alpha}} \right\}\leq p_t(y-x) \leq c_2 \left\{t^{-\alpha/d} \wedge \frac{t}{|x-y|^{d+\alpha}} \right\}.
\end{equation}

We now state some lemmas which help us to get the upper and lower bounds of the dimension. 
\begin{lemma}\label{lem:gaussian} Fix $t>0$. Then,
$\{Z_t(x)\}_{x\in\R^d}$ is a centered stationary Gaussian random field with variance $\cab\, t^{(\alpha-\beta)/\alpha}$, where $\cab$ is given in Theorem \ref{th:linear}. In addition,  there exist two constants $0<c_1=c_1(\alpha,\beta,t)\leq c_2=c_2(\alpha,\beta)<\infty$  such that, for all $x,y\in \R^d$ with $\|x-y\|<1$, 
\begin{equation}\label{eq:modulus:linear}
c_1 |x-y|^{\alpha-\beta}\leq \E|Z_t(x)-Z_t(y)|^2\leq c_2 |x-y|^{\alpha-\beta} \quad \text{for $t>0$}.
\end{equation} 
On the other hand, when $\|x-y\|>2 t^{1/\alpha}$,  there exists some constant $c_3:=c_3(\alpha,\beta, t)$ such that
\begin{equation}\label{eq:cov:linear}
0\leq \Corr(Z_t(x), Z_t(y) ) \leq c_3\|x-y\|^{-\beta}.
\end{equation}

\end{lemma}
\begin{proof}
First of all, it is clear that $Z_t(x)$ is centered stationary  Gaussian. In addition, by Walsh's isometry, change of variables and the Fourier transform, we have 
\begin{equation}
\begin{aligned}
\E Z_t^2 (x) &= \int_0^t  \int_{\R^d}  \int_{\R^d}   p_s(y) p_s(z) f(y-z)\, \d y \d z \d s \\
& = \int_0^t  \int_{\R^d} p_s(y) (p_s * f) (y) \, \d y  \d s \\
& = \int_0^t \int_{\R^d} \e^{-2s\|\xi\|^\alpha} \|\xi\|^{-d+\beta} \, \d \xi \d s\\
& = \cab\, t^{(\alpha-\beta)/\alpha},
\end{aligned}
\end{equation} which shows the first part of the theorem. 
We now show \eqref{eq:modulus:linear}: Using the Fourier transform and L\'evy exponent, we have
\begin{align*}
\E|Z_t(x)-Z_t(y)|^2&=2\left\{ \E Z^2_t(x) -\E (Z_t(x)Z_t(y)) \right\}\\
&=2\left\{ \int_0^t \int_{\R^d} \e^{-2s\|\xi\|^\alpha} \|\xi\|^{-d+\beta} \left( 1- \cos (\xi \cdot (x-y)\right) \,\d \xi \d s \right\}\\
&=2\left\{ \int_{\R^d} \frac{\left(1-\e^{-2t\|\xi\|^\alpha}\right)\|\xi\|^{-d+\beta}\left(1-\cos(\xi\cdot(x-y))\right)  }{2\|\xi\|^\alpha}  \d \xi\right\}\\
&=2A_1+2A_2,
\end{align*}where
\begin{align*}
A_1&:=\int_{\|\xi\|\leq \frac{1}{\|x-y\|}} \frac{\left(1-\e^{-2t\|\xi\|^\alpha}\right)\|\xi\|^{-d+\beta}\left(1-\cos(\xi\cdot(x-y))\right)  }{2\|\xi\|^\alpha}  \d \xi \\
A_2&:=\int_{\|\xi\|> \frac{1}{\|x-y\|}} \frac{\left(1-\e^{-2t\|\xi\|^\alpha}\right)\|\xi\|^{-d+\beta}\left(1-\cos(\xi\cdot(x-y))\right)  }{2\|\xi\|^\alpha}  \d \xi.
\end{align*}
We first consider $A_2$: Since $\beta-\alpha-1<-1$,
\begin{align}
A_2 \leq \int_{\|\xi\|> \frac{1}{\|x-y\|}} \|\xi\|^{-d+\beta-\alpha}\, \d \xi = \int_{\frac{1}{\|x-y\|}}^\infty r^{\beta-\alpha-1} \, \d r = \frac{\|x-y\|^{\alpha-\beta}}{\alpha-\beta}.
\end{align}
We now consider $A_1$: Since $|\xi\cdot (x-y)|\leq \|\xi\|\|x-y\| \leq 1$ and $-1\leq 1-\alpha<\beta-\alpha+1<1$, 
\begin{align}
A_1&\leq \|x-y\|^2  \int_{\|\xi\|\leq \frac{1}{\|x-y\|}} \|\xi\|^{-d+\beta-\alpha+2}\, \d \xi =\|x-y\|^2 \int_0^{1/\|x-y\|} r^{\beta-\alpha+1} \,\d r\\
&\leq \frac{\|x-y\|^{\alpha-\beta}}{2-\alpha+\beta}.
\end{align} 
For the lower bound, since $A_2\geq 0$, it is enough to consider $A_1$ only: 
\begin{align}
A_1&\geq \int_{\left\{\|\xi\|\leq \frac{1}{\|x-y\|}\right\} \cap \left\{\frac{1}{2}\leq \xi\cdot (x-y)\right\} } \frac{\left(1-\e^{-2t\|\xi\|^\alpha}\right)\|\xi\|^{-d+\beta}\left(1-\cos(\xi\cdot(x-y))\right)  }{2\|\xi\|^\alpha}  \d \xi \\
&\geq \text{const.} \left( 1-\e^{-2t/ \|x-y\|^\alpha} \right) \|x-y\|^{\alpha-\beta}.
\end{align} We used the fact that the area of the region $\left\{\|\xi\|\leq 1/ \|x-y\|\right\} \cap\{\xi \in \R^d: 1/2 \leq \xi\cdot (x-y)\}$ is comparable to the area of the ball $\{\xi \in \R^d: \|\xi\|\leq 1/\|x-y\| \}$. Combining things together, we complete the proof of \eqref{eq:modulus:linear}. 

Lastly, we show \eqref{eq:cov:linear}: Using the change of variables and the semigroup property, we have
\begin{align}
\E (Z_t(x)Z_t(y)) &= \int_0^t  \int_{\R^d}  \int_{\R^d}   p_s(x-z) p_s(y-z') f(z-z')\, \d z \d z' \d s \\
& = \int_0^t \int_{\R^d} p_{2s}(z) f(z-(x-y)) \, \d z \d s.
\end{align} 
Let $w:=x-y$. If $\|z-w\|\leq \|w\|/4$, then $\|z\|\geq \|w\|/2 \geq t^{1/\alpha}$. Thus, by \eqref{eq:density},  we have
\begin{equation}\label{eq:a1}
\begin{aligned}
&\int_0^t \int_{\|z-w\|\leq \|w\|/4} p_{2s}(z)\|z-w\|^{-\beta} \, \d z \d s\\
&\leq c_2 \int_0^t 2s \int_{\|z-w\|\leq \|w\|/4} \|z\|^{-d-\alpha} \|z-w\|^{-\beta} \, \d z \d s\\
&\leq c_2 2^{d+\alpha+1}  \int_0^t s \|w\|^{-d-\alpha} \int_{\|z-w\|\leq \|w\|/4}   \|z-w\|^{-\beta} \, \d z \d s\\
&\leq \frac{c_2 t^2  2^{-d+\alpha+2\beta} }{d-\beta} \|w\|^{-\alpha-\beta}.
\end{aligned}
\end{equation}  On the other hand, we have
\begin{equation}\label{eq:a2}
\begin{aligned}
&\int_0^t \int_{\|z-w\|> \|w\|/4} p_{2s}(z)\|z-w\|^{-\beta} \, \d z \d s\\
&\leq 4^\beta \|w\|^{-\beta} \int_0^t \int_{\R^d} p_{2s}(z) \d z \d s\\
&\leq 4^\beta t \|w\|^{-\beta}. 
\end{aligned}
\end{equation}
Combining \eqref{eq:a1} and \eqref{eq:a2},  we complete the proof of \eqref{eq:cov:linear}.
\end{proof}

Lemma \ref{lem:gaussian} implies that we can get  the following lemma which is originally due to Qualls and Watanabe \cite[Theorem 2.1]{QW}:
\begin{lemma}[\cite{QW}, Theorem 2.1]\label{lem:linear:tail}
Let $D$ be a compact set in $\R^d$. There exists constant $H_{\alpha-\beta} \in (0,\infty)$ such that
\begin{equation}
\lim_{\lambda \to \infty} \P\left\{ \sup_{x\in D} \frac{Z_t(x)}{\cab t^{(\alpha-\beta)/2\alpha})} >\lambda \right\}\Big/  \psi(\lambda) \left(\tilde{\sigma}^{-1}(1/\lambda)\right)^{-d} = H_{\alpha-\beta},
\end{equation}where $|D|$ is the Lebesgue measure of $D$, $\psi(x)=\exp(-x^2/2)/(2\pi x)$, and $\tilde{\sigma}^2(s)=2s^{\alpha-\beta}$ for $s\geq 0$. 
\end{lemma}
%One can check easily that Lemma \ref{lem:gaussian} verifies the condition \eqref{cond:UB}.  We  give a proof how the condition \eqref{cond:UB} implies the upper bound. 
\begin{proof}[Proof of the upper bound in Theorem \ref{th:linear}]
First, we note that since $\{Z_t(x)\}_{x\in\R^d}$ is a Gaussian random field with variance $\cab\, t^{(\alpha-\beta)/\alpha}$ by Lemma \ref{lem:gaussian},  if we let $X(x):=Z_t(x)/ \sqrt{\cab\, t^{(\alpha-\beta)/\alpha}}$, then  it is easy to see that $b=2$ and $\bm{c}(2)=\bm{C}(2)=1/2$ in Section \ref{sec:general}. In addition, one can easily check that the condition \eqref{cond:UB} is satisfied by Lemma \ref{lem:linear:tail}, which gives us the upper bound we want. Here, we explain how the condition \eqref{cond:UB} implies the upper bound, which helps readers understand why the condition \eqref{cond:UB} is quite a natural condition. 

Recall 
\[ \cP_{Z_t}(\gamma):=\left\{ x\in\R^d: \, \|x\|>\e, Z_t(x)\geq \left(2\, \cab\, t^{(\alpha-\beta)/\alpha}\, \gamma\, \log\|x\| \right)^{1/2} \right\}. \] We show that the macroscopic Hausdorff dimension of $\cP_{Z_t}(\gamma)$ is bounded above by $d-\gamma$.  
Define, for all $n\geq 0$,  
\begin{align*}
&\bar\cP_{Z_t}^{(n)}(\gamma):=\left\{ x\in \cS_n: \, \|x\|>\e, \sup_{y\in Q(x,1)} Z_t(y)\geq \left(2\, \cab\, t^{(\alpha-\beta)/\alpha}\, \gamma\, n \right)^{1/2} \right\}, \\
&\bar\cP_{Z_t}(\gamma):=\cup_{n=0}^\infty \, \bar\cP_Z^{(n)}(\gamma).
\end{align*}
Since $\cP_{Z_t}(\gamma) \subset \bar\cP_{Z_t}(\gamma)$,  covering $\bar\cP_{Z_t} \cap \cS_n$ by upright boxes $Q(x,1)$ for $x\in \cS_n\cap \Z^d$, we get that
\begin{equation}\label{eq:prob-est}
\begin{aligned}
\E\, \nu_\rho^n\left(\cP_{Z_t} (\gamma)\right) &\leq \E \sum_{x\in \cS_n\cap \Z^d} \left(\frac{1}{\e^n}\right)^\rho \mathbb{1}_{\{x\in \bar\cP_{Z_t}(\gamma)\}}(x)\\
&\leq \text{const.}\times\e^{n(d-\rho)}\sup_{x\in \cS_n} \P\left\{\sup_{y\in Q(x,1)} Z_t(y) \geq \sqrt{2\cab t^{(\alpha-\beta)/\beta} \gamma \log n } \right\}\\
&\leq \text{const.}\times \e^{n\left(d-\gamma-\rho+o(1)\right)} \quad \text{as $n\to \infty$}.
\end{aligned}
\end{equation} For the last inequality above, we used Lemma \ref{lem:linear:tail}.  Therefore, for all $\rho>d-\gamma$, we have
$\sum_{n=1}^\infty \E\, \nu_\rho^n\left(\cP_{Z_t} (\gamma)\right) <\infty$, 
which implies that
\[ \Dimh \cP_{Z_t}(\gamma) \leq \rho, \quad a.s..\]
Since $\rho>d-\gamma$ is arbitrary, we get the upper bound of the dimension. For \eqref{eq:sup:linear}, it can be easily derived from a similar probability estimate as in \eqref{eq:prob-est} and the Borel-Cantelli lemma.
\end{proof}

We now consider the lower bound. 
\begin{proof}[Proof of the lower bound in Theorem \ref{th:linear}]\label{proof:lower}
We first construct independent random variables which can provide a probability estimate as in \eqref{eq:lower}. 

Let $x\in \Pi_n(\theta)$. Let $\{y_i\}_{i=1}^m$'s be the set of the elements in $\Pi_n(\delta)\cap Q(x,\e^{n\theta})$. Thus, there exists some constant $c>1$ such that $c^{-1}\e^{n(\theta-\delta)}\leq m\leq c\e^{n(\theta-\delta)}$, i.e., $m\asymp \e^{n(\theta-\delta)}$. Let $C:=(c_{i,j})_{m\times m}$ be a matrix whose elements are given by $c_{i,i}=1$ and $c_{i,j}=r_n$ for all $1\leq i\neq j \leq m$, where
\[ r_n:=\max_{\substack{y_i\neq y_j\\ y_i, y_j \in \Pi_n(\delta)\cap Q(x,\e^{n\theta})}} \Corr(Z_t(y_i), Z_t(y_j)), \quad \text{for a fixed $t>0$.}\] 
By \eqref{eq:cov:linear}, we have $0< r_n \leq c_3\, \exp\left(-\beta\delta n\right)$.
Since the matrix $C$ is positive definite, we can construct Gaussian random variables $\{\tilde Z_i\}_{i=1}^m$ whose covariance is given by the matrix $C$. In addition, since these random variables are equally correlated, by Theorem in Section 2 in \cite{Berman}, there exist centered independent Gaussian random variables $Y, U_1, U_2, \dots, U_m$ such that 
\begin{align*}
&\tilde Z_i=Y+U_i, \\
\E Y^2=r_n, &\quad \E U_i^2=1-r_n, \,\, 1\leq i\leq m.
\end{align*}
Therefore, using the tail probability estimate of a centered Gaussian random variable, we obtain that for all large $n\geq 1$, 
\begin{equation}
\begin{aligned}\label{linear:max:ineq}
\P\left\{\max_{1\leq i\leq m} \tilde Z_i \leq \sqrt{2\gamma n} \right\} & \leq \P\left\{\max_{1\leq i\leq m} U_i \leq \sqrt{2\gamma n+2} \right\} +\P\left\{Y \geq \sqrt{2\gamma n+1}-\sqrt{2\gamma n} \right\} \\
& \leq \left(\P\left\{U_i\leq \sqrt{2\gamma n+2} \right\} \right)^m+\P\left\{Y \geq \frac{1}{\sqrt{9\gamma n}} \right\} \\
&\leq \left(1-\exp\left( -\frac{\gamma n}{1-r_n}\right) \right)^m + \exp\left(-\frac{1}{18\gamma n\, r_n}\right)\\
&\leq \exp\left[-m\times \e^{-\gamma n +o(1)}\right]+ \exp\left(-\frac{c_3\e^{\beta\delta n}}{18\gamma n}\right)\\
&\leq \exp\left[-\text{const.}\times \e^{n(d\theta-d\delta-\gamma) +o(1)}\right]+ \exp\left(-\frac{c_3\e^{\beta\delta n}}{18\gamma n}\right).
\end{aligned}
\end{equation} Here, we used the fact that $m\asymp \e^{nd(\theta-\delta)}$ and $r_n\leq c_3\exp(-\beta\delta n)$. We also note that the const. in the above inequality does not depend on $n$. 

Let us now get a probability estimate as in \eqref{eq:lower}. First, we note that from the construction     of Gaussian random variables $\{\tilde Z_i\}_{i=1}^m$, we have $\Corr(\tilde Z_i, \tilde Z_j) \geq \Corr(Z_t(y_i), Z_t(y_j))$ for all $1\leq i, j \leq m$. Thus,  Slepian's inequality \cite{Slepian} says that for any $\lambda>0$ 
\begin{equation}\label{eq:slepian}
\P\left\{\max_{\{y_i\}_{i=1}^m \in \Pi_n(\delta)\cap  Q\left( x,\e^{n\theta}\right)}\frac{Z_t(y_i)}{ \sqrt{ \cab\, t^{(\alpha-\beta)/\alpha}} } \leq \lambda \right\} \leq  \P\left\{\max_{1\leq i\leq m} \tilde Z_i \leq \lambda \right\}. 
\end{equation}
Therefore, using \eqref{linear:max:ineq} and \eqref{eq:slepian},  we obtain 
\begin{equation*}
\begin{aligned}
 &\P \left\{\max_{\{y_i\}_{i=1}^m \in \Pi_n(\delta)\cap  Q\left( x,\e^{n\theta}\right)}\frac{Z_t(y_i)}{(\log\|y_i\|)^{1/2}}< \sqrt{2 \cab\, t^{(\alpha-\beta)/\alpha} \, \gamma}  \right\}  \\ 
&\leq  \P \left\{\max_{\{y_i\}_{i=1}^m \in \Pi_n(\delta)\cap  Q\left( x,\e^{n\theta}\right)}\frac{Z_t(y_i)}{ \sqrt{ \cab\, t^{(\alpha-\beta)/\alpha}} }< \sqrt{2 \gamma n}  \right\}  \\ 
&\leq \exp\left[-\text{const.}\times \e^{n(d\theta-d\delta-\gamma) +o(1)}\right]+ \exp\left(-\frac{c_2\e^{\beta\delta n}}{18\gamma n}\right).
 \end{aligned} 
\end{equation*}
Now we have the same procedure as in the proof of Theorem \ref{th:Gen:LIL:LB}, i.e., 
\begin{align*}
&\sum_{n=0}^\infty \sum_{x\in \Pi_n(\theta)} \P \left\{\max_{\{y_i\}_{i=1}^m \in \Pi_n(\delta)\cap  Q\left( x,\e^{n\theta}\right)}\frac{Z_t(y_i)}{(\log\|y_i\|)^{1/2}}< \sqrt{2 \cab\, t^{(\alpha-\beta)/\alpha} \, \gamma}  \right\}  \\ 
&\leq \sum_{n=0}^\infty \exp(nd(1-\theta)) \left(\exp\left[-\text{const.}\times \e^{n(d\theta-d\delta-\gamma) +o(1)}\right]+ \exp\left(-\frac{c_2\e^{\beta\delta n}}{18\gamma n}\right)
 \right)\\
&< \infty 
\end{align*} as long as $d\theta -d\delta-\gamma>0$, equivalently, $\theta> \gamma/d + \delta$. 
By the Borel-Cantelli lemma, Proposition \ref{pr:thick}, and letting $\delta$ to 0, we get the desired lower bound.
\end{proof}

\section{Parabolic Anderson model for fractional Laplacian}\label{sec:FSHE}

\subsection{Tail probabilities  and the upper bound in Theorem \ref{th:pam}}

Let us now consider the following parabolic Anderson model for fractional Laplacian:
\begin{equation}
\begin{aligned}\label{eq:pam}
\frac{\partial  }{\partial t} u_t(x)&= -(-\Delta)^{\alpha/2}\, u_t(x)+u_t(x) \dot{F}(t,x), \quad t>0, x\in \R^d,\\
u_0(x)&=1, \quad x\in \R^d.
\end{aligned}
\end{equation} 
Throughout this section, we define $\|X\|_k:=\{\E|X|^k\}^{1/k}$ for a random variable $X$. Let us first consider the high moment asymptotics which lead to the asymptotic tail probability.

\begin{lemma}\label{lem:moment}
There exists constants $0<c_*=c_*(\alpha,\beta, d) \leq c^*=c^*(\alpha,\beta,d) < \infty$ such that
\begin{equation}
0< c_*t \leq \liminf_{k\to \infty} \frac{\log \E[u_t(x)]^k}{k^{(2\alpha-\beta)/(\alpha-\beta)}} \leq \limsup_{k\to \infty} \frac{\log \E[u_t(x)]^k}{k^{(2\alpha-\beta)/(\alpha-\beta)}} \leq c^* t <\infty
\end{equation} for all $x\in\R^d$. 
\end{lemma}

\begin{proof}
We first consider the upper bound. Let $k\geq 2$.  By the Burkholder-Davis-Gundy inequality (\cite[Theorem B.1]{cbms}), we have
\begin{equation}\label{eq:BDG}
\begin{aligned}
&\left\|\int_{(0,t)\times \R^d} p_{t-s}(y-x) u_s(y) F(\d s \d y)\right\|_k^k \\
&\leq (4k)^{k/2} \left\| \int_0^t \int_{\R^d} \int_{\R^d} p_{t-s}(y-x) p_{t-s}(y'-x) u_s(y)u_s(y')f(y-y')\d y \d y' \d s  \right\|^{k/2}_{k/2} \\
& \leq (4k)^{k/2} \left\{\int_0^t \int_{\R^d}\int_{\R^d} p_{t-s}(y-x)p_{t-s}(y'-x) \|u_s(y)u_s(y')\|_{k/2} f(y-y') \d y \d y' \d s \right\}^{k/2}\\
& \leq (4k)^{k/2} \left\{\int_0^t  \sup_{x\in \R^d}\|u_s(x)\|_k^2 \int_{\R^d}\int_{\R^d} p_{t-s}(y-x)p_{t-s}(y'-x) f(y-y') \d y \d y' \d s \right\}^{k/2}
\end{aligned}\end{equation} In the last inequalities above, we used Minkowski's integral inequality and Cauchy-Schwartz inequality. Thus, using the Fourier transform, we have
\begin{equation}\label{eq:BDG1}
\begin{aligned}
\|u_t(x)\|_k^2 
&\leq 2\left\{1+4k \int_0^t \sup_{x}\|u_s(x)\|_k^2 \int_{\R^d}\int_{\R^d} p_{t-s}(y) p_{t-s}(y') f(y-y') \d y \d y' \d s\right\}\\
&\leq 2\left\{1+4k\int_0^t \sup_{x}\|u_s(x)\|_k^2 \int_{\R^d} |\hat{p}_{t-s}(\xi)|^2  \hat{f}(\xi) \d \xi\d s \right\}\\
&\leq 2\left\{1+4k\int_0^t \sup_{x}\|u_s(x)\|_k^2 \int_{\R^d} \e^{-2(t-s)\|\xi\|^\alpha} \|\xi\|^{-d+\beta} \d \xi \d s\right\}\\
&\leq 2\left\{1+2^{(2-\beta)/\alpha}\, \Gamma(\beta/\alpha) k \int_0^t \frac{\sup_{x} \|u_s(x)\|_k^2}{(t-s)^{\beta/\alpha}} \d s \right\}.
\end{aligned}
\end{equation}
For any random field $\Phi_t(x)$, any $\gamma>0$ and $k\in [2,\infty)$, we  define, as in \cite[(5.4)]{cbms},
\begin{equation}\label{norm}
\mathcal{N}_{\gamma, k}(\Phi):=\sup_{t\geq 0}\sup_{x\in\R^d} \e^{-\gamma t} \|\Phi_t(x)\|_k^2.
\end{equation}
Using \eqref{norm}, we obtain
\begin{align*}
\mathcal{N}_{\gamma,k}(u) &\leq 2\left\{1+ 2^{(2-\beta)/\alpha}\, \Gamma(\beta/\alpha)\, k\,  \mathcal{N}_{\gamma, k}(u) \int_0^\infty \e^{-\gamma s} s^{-\beta/\alpha} \d s \right\}\\ 
&\leq 2+\bar c^*\, k \, \gamma^{(\beta-\alpha)/\alpha} \mathcal{N}_{\gamma, k}(u),
\end{align*} where $\bar c^*:=2^{(3-\beta)/\alpha}\Gamma(\beta/\alpha) \Gamma(1-\beta/\alpha)$.
We now choose $\gamma=\left( 2\bar c^* k \right)^{\alpha/(\alpha-\beta)}$ to obtain 
\begin{equation}\label{eq:upper-moment}
 \sup_{x\in\R^d} \|u_t(x)\|_k^2 \leq 4\e^{t(2\bar c^*k)^{\alpha/(\alpha-\beta)}}, \quad t>0. 
 \end{equation} 
This shows the upper bound where $c^*:=2^{-1}(2\bar c^*)^{\alpha/(\alpha-\beta)}$. Note that the bound in \eqref{eq:upper-moment} is for all $k\geq 2$ and it will be used later on. 

Let us now consider the lower bound. We use the Feynman-Kac formula for the moment (see \cite{Conus, HuNualart}), i.e., 
\begin{align*}
\E [u_t(x)]^k &= \E_x \left[ \exp\left(\sum_{\substack{1\leq i, j, \leq k \\ i\neq j }} \int_0^t f \left( X_s^{(i)}-X_s^{(j)} \right) \d s \right) \right]\\
&=\E_x \left[ \exp\left(\sum_{\substack{1\leq i, j, \leq k \\ i\neq j }} \int_0^t \frac{c_{\beta, d}}{\left\|X_s^{(i)}-X_s^{(j)}\right\|^\beta} \d s \right) \right],
\end{align*}where the processes $\{X^{(i)}\}_{i=1}^k$ are $k$ independent copies of a symmetric stable process whose generator is $-(-\Delta)^{\alpha/2}$ and $\E_x$ is the expectation with respect to the law of these processes conditioned on $X^{(i)}=x$ for $1\leq i\leq k$. 

Let $A_\epsilon(t):=\left\{ \omega \in \Omega: \max_{1\leq l\leq k} \sup_{0\leq s\leq t}\|X^{(l)}_s\| \leq \epsilon \right\}$ for $\epsilon>0$. Then, Proposition 3 in Chapter 8 of \cite{Bertoin} tells us that there exists some constant $c>0$ such that for all small $\epsilon>0$ 
\begin{equation}
\P(A_\epsilon(t))= \left(\P\left\{\sup_{0\leq s\leq t}\|X^{(1)}_s\| \leq \epsilon \right\}\right)^k \geq \exp\left(-ck\epsilon^{-\alpha} t\right).
\end{equation} 
Thus, we have
\begin{align*}
\E[u_t(x)]^k &\geq \exp\left[ c_{\beta,d}\, k(k-1)t(2\epsilon)^{-\beta} \right]  \E_x \mathds{1}_{\left\{\max_{1\leq l\leq k} \sup_{0\leq s\leq t}\|X^{(l)}_s\| \leq \epsilon\right\}}\\
&\geq \exp\left[c_{\beta,d}\,  k(k-1)t(2\epsilon)^{-\beta} -ck\epsilon^{-\alpha} t \right].
\end{align*} 
We now maximize the last term in the above inequality over $\epsilon$ to get the lower bound of the moment, i.e., 
\begin{equation}
E[u_t(x)]^k \geq \exp\left( c_*k^{(2\alpha-\beta)/(\alpha-\beta)} t\right),
\end{equation} for some constant $c_*:=c_*(\alpha,\beta, d)$ and for all large $k\geq 1$ (the maximum occurs at $\epsilon\approx k^{-1/(\alpha-\beta)}$). 
This completes the proof. 
\end{proof}

We can now get some estimates of the asymptotic tail probability:
\begin{lemma}
Let $t>0$. Then, uniformly for all $x\in\R^d$, we have
\begin{align}
 \liminf_{z \to \infty} \frac{\log \P\{ \log u_t(x) \geq z\} }{z^{(2\alpha-\beta)/\alpha}} & \geq -{\bm{C}} t^{(\beta-\alpha)/\alpha}, \\
  \limsup_{z \to \infty}\frac{ \log \P\{ \log u_t(x) \geq z\}}{z^{(2\alpha-\beta)/\alpha}} &\leq -{\bm{c}}t^{(\beta-\alpha)/\alpha}, \label{eq:tail:upper}
\end{align}where
\begin{align*}
{\bm{c}}&:=\left(\frac{\alpha-\beta}{2\alpha-\beta}\right)^{(\alpha-\beta)/\alpha}\left[1- \left(\frac{\alpha-\beta}{2\alpha-\beta}\right)^{(2\alpha-\beta)/(\alpha-\beta)}\right](c^*)^{-(\alpha-\beta)/\alpha}, \\
{\bm{C}}&:=\left(\frac{\alpha-\beta}{2\alpha-\beta}\right)^{(\alpha-\beta)/\alpha}\left[1- \left(\frac{\alpha-\beta}{2\alpha-\beta}\right)^{(2\alpha-\beta)/(\alpha-\beta)}\right](c_*)^{-(\alpha-\beta)/\alpha} 
\end{align*} wherein $c^*$ and $c_*$ are the constants in Lemma \ref{lem:moment}
\end{lemma}
\begin{proof}
In \cite[Theorem 5.4]{Chen}, Chen considers the asymptotic tail probability of the solution to \eqref{eq:pam} when $\alpha=2$. He obtained the asymptotic tail probability from the moment asymptotics by using the G\"{a}rtner-Ellis theorem. Since there is no role of Laplacian ($\alpha=2$ in our case \eqref{eq:pam}) in the proof of \cite[Theorem 5.4]{Chen} and we also have the moment asymptotics by Lemma \ref{lem:moment},  we can follow almost exactly the proof of \cite[Theorem 5.4]{Chen}. Thus, we we skip the proof. We just  note that the constants ${\bm{c}} t^{(\alpha-\beta)/\alpha}$ and ${\bm{C}}t^{(\alpha-\beta)/\alpha}$ can be obtained from maximizing the function   $g(k):= k-ctk^{(2\alpha-\beta)/(\alpha-\beta)}$ where $c$ can be $c^*$ or $c_*$.  
\end{proof}

\begin{proof}[Proof of the upper bound in Theorem \ref{th:pam}]
We will verify the condition \eqref{cond:UB} given in Theorem \ref{th:upper:general}. 

Let $q_t(x,y;z):=p_t(x-z)-p_t(y-z)$.   Since we have
\begin{align*}
u_t(x)-u_t(y)=\int_{(0,t)\times \R^d} q_{t-s}(x,y;z) u_s(z) F(\d s \, \d z),
\end{align*}
following the same computations as in \eqref{eq:BDG} and \eqref{eq:BDG1}, and also by \eqref{eq:upper-moment} (after relabeling $c^*$ in \eqref{eq:upper-moment}), we obtain 
\begin{align*}
&\|u_t(x)-u_t(y)\|_k^2 \leq 4k \int_0^t \sup_{x\in\R^d}\|u_s(x)\|_k^2 \int_{\R^d} \int_{\R^d} q_{t-s}(x,y;z)q_{t-s}(x,y,z') f(z-z') \d z \d z' \d s\\
&\qquad \qquad\leq 4k \int_0^t 4\exp\left(2c^*s   k^{\alpha/(\alpha-\beta)}\right)  \int_{\R^d} \int_{\R^d} q_{t-s}(x,y;z)q_{t-s}(x,y,z') f(z-z') \d z \d z' \d s.
\end{align*}
Here, we observe that 
\[
 \int_{\R^d} \int_{\R^d} q_{t-s}(x,y;z)q_{t-s}(x,y,z') f(z-z') \d z \d z' \d s=\E|Z_{t-s}(x)-Z_{t-s}(y)|^2 \leq c_2\|x-y\|^{\alpha-\beta},
\]where the constant $c_2$ does not depend on $t-s$ by Lemma \ref{lem:gaussian}.  Hence, for all large enough $k\geq 2$, there exists some contant $\bar{c} > c^*$ such that
\begin{equation*}
\|u_t(x)-u_t(y)\|_k^k \leq \exp\left(\bar{c} t k^{(2\alpha-\beta)/(\alpha-\beta)} \right) \|x-y\|^{(\alpha-\beta)k/2}.
\end{equation*} This and a quantitative form of the Kolmogorov continuity theorem (see e.g. \cite[Theorem C.6]{cbms}) imply that there exists some contant $0<\tau:=\tau(\alpha,\beta, d, t)<\infty$ such that 
\begin{equation*}
\sup_{z\in \R^d}  \E\left[
		\sup_{\substack{x, y\in Q(z,1)\\
		x\neq y}} \left[\frac{|u_t(x)-u_t(y)|}{|x-y|^{(\alpha-\beta)/4}}\right]^k
		\right]< \tau \exp\left(\tau k^{(2\alpha-\beta)/(\alpha-\beta)} t\right).
\end{equation*} 
Thus, for any $\eta, \epsilon \in (0,1)$, we have
\begin{equation}\label{eq:BD1}\begin{split}
		&\sup_{y\in\R^d}\P\left\{\sup_{x\in Q(y,\epsilon)}|u_t(x)-u_t(y)|>
			\exp\left[(\eta\log s)^{\alpha/(2\alpha-\beta)}\right]\right\}\\
			&\qquad \le \tau \varepsilon^{(\alpha-\beta)k/4}
			\, \exp\left(\tau k^{(2\alpha-\beta)/(\alpha-\beta)} t - k(\eta\log s)^{\alpha/(2\alpha-\beta)}\right).
	\end{split}\end{equation}
Let us choose 
\begin{align*}
\epsilon:=\epsilon(s):=\exp\left(-\lambda_1 (\log s)^{\alpha/(2\alpha-\beta)}\right)\quad,  \qquad k:=k(s):=\lambda_2 (\log s)^{(\alpha-\beta)/(2\alpha-\beta)},
\end{align*}where
\[
\lambda_1:=\frac{8\gamma (\tau t)^{(\alpha-\beta)/\alpha}}{(\alpha-\beta)\eta^{(\alpha-\beta)/(2\alpha-\beta)}} \quad , \qquad \lambda_2:=\frac{\eta^{(\alpha-\beta)/(2\alpha-\beta)}}{(\tau t)^{(\alpha-\beta)/\alpha}}.
\]
In this way, we obtain 
\begin{equation*}
\sup_{y\in\R^d}\P\left\{\sup_{x\in Q(y,\epsilon)}|u_t(x)-u_t(y)|>
			\exp\left[(\eta\log s)^{\alpha/(2\alpha-\beta)}\right]\right\} \leq \tau s^{-2\gamma}
\end{equation*} for all large $s>0$. 

We now choose $\eta:=\min\left\{\frac{1}{2}, \frac{\gamma}{2 \bm{c} t^{(\beta-\alpha)/\alpha}}  \right\}$. Thus, by \eqref{eq:tail:upper}, we get
\begin{equation*}
\sup_{x\in\R^d}\P\left\{u_t(x) \geq \exp\left[\left( \frac{\gamma}{\bm{c}\, t^{(\beta-\alpha)/\alpha}} \log s \right)^{\alpha/(2\alpha-\beta)} \right]- \exp\left[\left(\eta \log s \right)^{\alpha/(2\alpha-\beta)}\right] \right\} \leq s^{-\gamma + o(1)}.
\end{equation*}
Finally, since we have
\[\lim_{s\to \infty} \frac{\log\left[1+\frac{1}{\epsilon(s)}\right]}{\log s}=0,\] 
we obtain
\begin{equation*}
\begin{aligned}
&\sup_{y\in\R^d} \P\left\{\sup_{x\in Q(y,1)} \log u_t(x) > \left( \frac{\gamma}{\bm{c}\, t^{(\beta-\alpha)/\alpha}} \log s \right)^{\alpha/(2\alpha-\beta)} \right\}\\
&=\sup_{y\in\R^d} \P\left\{\sup_{x\in Q(y,1)} u_t(x) > \exp\left[\left( \frac{\gamma}{\bm{c}\, t^{(\beta-\alpha)/\alpha}} \log s \right)^{\alpha/(2\alpha-\beta)}\right] \right\}\\
&\leq \left(1+\frac{1}{\epsilon(s)}\right)^d \sup_{y\in\R^d}\P\left\{\sup_{x\in Q(y,\epsilon)} u_t(x) > \exp\left[\left( \frac{\gamma}{\bm{c}\, t^{(\beta-\alpha)/\alpha}} \log s \right)^{\alpha/(2\alpha-\beta)}\right] \right\}\\
&\leq \left(1+\frac{1}{\epsilon(s)}\right)^d\sup_{y\in\R^d} \P\left\{u_t(y) \geq \exp\left[ \left( \frac{\gamma}{\bm{c}\, t^{(\beta-\alpha)/\alpha}} \log s \right)^{\alpha/(2\alpha-\beta)}\right] - \exp\left[\left(\eta \log s \right)^{\alpha/(2\alpha-\beta)} \right]\right\} \\
&\qquad + \left(1+\frac{1}{\epsilon(s)}\right)^d\sup_{y\in\R^d}\P\left\{\sup_{x\in Q(y,\epsilon)}|u_t(x)-u_t(y)|>\exp\left[(\eta\log s)^{\alpha/(2\alpha-\beta)}\right]\right\}\\
&\leq s^{-\gamma+o(1)}		
\end{aligned}
\end{equation*} as $s\to \infty$. This verifies \eqref{cond:UB}.
\end{proof}

\subsection{The lower bound in Theorem \ref{th:pam}}\label{sec:lower:pam}
We now consider the lower bound in Theorem \ref{th:pam}. Note that, differently from the proof of the lower bound in Theorem \ref{th:linear} where we used the properties of Gaussian random variables, we verify the condition \eqref{cond:LB} by first constructing a coupling process close to the solution $u_t(x)$. To do this, we use the coupling of noise and localization argument developed in \cite{CJKS}.

Let $\eta$ be space-time white noise on $\R_+\times \R^d$ and $h(x):=\left(c_{\beta,d}\right)^{1/2}\|x\|^{-(d+\beta)/2}$ for $x\in\R^d$, where $c_{\beta,d}$ is defined in  \eqref{eq:Riesz}. We note that $(h \ast h) (x)=f(x)$ so that $|\hat h|^2=\hat f$.  Here,  the convolution and the Fourier transform can be understood as the convolution and Fourier transform of generalized functions. We define, for every $\phi \in \mathcal{S}$ (the Schwartz space, i.e., the space of all test functions of rapid decrease), 
\begin{equation} \label{eq:coupling}
 F_t^{(h)}(\phi):=\int_{(0,t)\times \R^d} (\phi \ast h)(x) \, \eta(\d s, \d x),
 \end{equation}where the stochastic integral can be understood in the sense of Walsh or Dalang (\cite{Walsh, Dalang}). By the Fourier and inverse Fourier transforms, it is easy to see that $\{F_t^{(h)}(\phi)\}_{\phi\in\mathcal{S}}$ is a centered Gaussian field whose covariance is given by 
 \begin{align*}
 \Cov\left(F_t^{(h)}(\phi), F_t^{(h)}(\psi)\right)&=t\int_{\R^d} (\phi\ast h)(x) (\psi \ast h)(x)\,  \d x \\
& =t\int_{\R^d} \hat{\phi}(\xi) \Bar{\hat{\psi}}(\xi) |\hat h(\xi)|^2\, \d \xi\\
&=t\int_{\R^d} \int_{\R^d} \phi(x) \psi(y) f(x-y)\, \d y \d x\\
&=\Cov \left(F_t(\phi), F_t(\psi)\right).
\end{align*}
Thus, we can regard $F$ as $F^{(h)}$ and vice versa (see Section 3.1 and 3.2 in \cite{CJKS} for a general construction of noise coupling for functions in $W^{1,2}_{loc}(\R^d)$ where $W^{1,2}_{loc}(\R^d)$ denotes the vector space of all locally integrable functions $g: \R^d \to  \R$ whose Fourier transform is a function that   satisfies  $\int_{\|x\|\leq r} \|x\|^2 |\hat{g} (x)|^2 \, \d x <\infty$). In particular, we can regard $u_t(x)$  as 
\begin{align}
u_t(x)&=1+\int_{(0,t)\times \R^d} p_{t-s}(y-x) u_s(y) F^{(h)}(\d s\, \d y)\\
&= 1+\int_{(0,t)\times \R^d} \left(\int_{\R^d} p_{t-s}(y-x) u_s(y) h(y-z) \, \d y \right) \, \eta(\d s\, \d z).\label{eq:zt}
\end{align} 
We now use a truncation to produce independence in space. For every $x:=(x_1,\dots, x_d)\in \R^d$, $t>0$, and $\ell > 0$, we define
 \begin{equation*} 
h_\ell(x):=h(x) \mathbb{1}_{\{\|x\| \leq \ell\}} (x),
 \end{equation*} 
 and
\begin{equation*} 
	\mathcal{I}_t(x;\ell):=
 \left[x_1-\ell t^{1/\alpha}\,, x_1+\ell t^{1/\alpha}\right]\times \cdots \times 
	\left[x_d-\ell t^{1/\alpha}\,, x_d+\ell t^{1/\alpha}\right].
\end{equation*}

Since $h_\ell  \in W^{1,2}_{loc}(\R^d)$,  replacing $h$ by $h_\ell$ in \eqref{eq:coupling} , we can also construct $F^{(h_\ell)}$. Using $F^{(h_\ell)}$, we may obtain a random field $u_t^{(\ell)}(x)$ which satisfies the following integral equation: 
 \begin{align}\label{eq:ul}
 \ul_t(x)&=1+ \int_{(0,t)\times \mathcal{I}_t(x;\ell)} p_{t-s}(y-x) \ul_s(y) \, F^{(h_\ell)}(\d s\, \d y)\\
 &=\int_{(0,t)\times \R^d} \left(\int_{\mathcal{I}_t(x;\ell)} p_{t-s}(y-x) \ul_s(y) h_\ell(y-z) \, \d y \right) \, \eta(\d s\, \d z).\label{eq:ztl}
\end{align}
The existence and uniqueness of $u_t^{(\ell)}(x)$ can be obtained from a Picard iteration, i.e., define iteratively,
\begin{align}
&u^{(\ell,0)}_t(x):=1,\\
&u^{(\ell, m)}_t(x):=1+\int_{(0,t)\times  \mathcal{I}_t(x;\ell)} p_{t-s}(y-x) u_s^{(\ell, m-1)} \, F^{(h_\ell)}(\d s\, \d y), \quad m\geq 1.
\end{align}
The following lemma says that there is a unique $\{\ul_t(x)\}_{t>0, x\in\R^d}$ which satisfies \eqref{eq:ul}.

\begin{lemma}\label{lem:picard}
Fix $t>0$. There exists a finite constant $C:=C(\alpha,\beta,t, d)>0$ such that for all $k\geq 2$
\begin{equation}
\sup_{x\in\R^d} \E\left| u_t^{(\ell, m+1)}(x) -u_t^{(\ell,m)}(x) \right|^k \leq C^k\,  2^{-m k/2} \,\e^{c^*k^{(2\alpha-\beta)/(\alpha-\beta)} t},
\end{equation}where $c^*$ is the constant given in Lemma \ref{lem:moment}
\end{lemma}
\begin{proof}
The proof of this lemma is very similar to the proof for the upper bound in Lemma \ref{lem:moment}. Thus, we may skip some steps. As in \eqref{eq:BDG} and \eqref{eq:BDG1}, we obtain
\begin{align}
\left\|u_t^{(\ell, m+1)}(x)-u_t^{(\ell,m)}(x)\right\|_k^2 \leq 2^{(2-\beta)/\alpha}\, \Gamma(\beta/\alpha) k \int_0^t (t-s)^{-\beta/\alpha} \sup_{x\in \R^d} \left\|u_t^{(\ell, m)}(x)-u_t^{(\ell,m-1)}(x)\right\|_k^2 \, \d s .
\end{align}
We also use the norm defined in \eqref{norm} and the same $\gamma$ given right after \eqref{norm} to obtain
\begin{equation}
\mathcal{N}_{\gamma, k}\left( u_t^{(\ell, m+1)}(x)-u_t^{(\ell,m)}(x)\right) \leq \frac{1}{2} \mathcal{N}_{\gamma, k}\left( u_t^{(\ell, m)}(x)-u_t^{(\ell,m-1)}(x)\right).
\end{equation}
Since $u_t^{(\ell, 1)}(x)-u_t^{(\ell,0)}(x)=Z_t(x)$ and $Z_t(x)$ is centered Gaussian with variance $\cab t^{(\alpha-\beta)/\alpha}$ by Lemma \ref{lem:gaussian}, we complete the proof. 
\end{proof}

Following the proof of Lemma \ref{lem:picard}, we can also have that 
\begin{equation}
\sup_{x\in\R^d} \E\left| u_t^{(\ell)}(x) -u_t^{(\ell,m)}(x) \right|^k \leq C^k\,  2^{-m k/2} \,\e^{c^*k^{(2\alpha-\beta)/(\alpha-\beta)} t}.
\end{equation}

\begin{lemma}\label{lem:indep:she}
Fix $t>0$, $\ell>1$ and $m\geq 1$. If $\|x-y\|_\infty \geq 2m(\ell t^{1/\alpha}+\ell)$ for $x, y\in\R^d$, then $u_t^{(\ell,m)}(x)$ and $u_t^{(\ell,m)}(y)$ are independent.   
\end{lemma}

We will skip the proof of Lemma \ref{lem:indep:she} since the proof is exactly the same as the one for Lemma 9.8 in \cite{CJKS} where the authors consider  \eqref{eq:pam} when $\alpha=2$. There is no difference between $\alpha=2$ and $0<\alpha<2$ in the proof of Lemma 9.8 in \cite{CJKS}.

\begin{lemma}\label{lem:u-ul}
Fix $t>0$ and $\ell >1$. There exists a finite constant $C:=C(\alpha,\beta, t, d)$ such that 
\begin{equation}
\sup_{x\in\R^d} \E \left| u_t(x)-u_t^{(\ell)}(x) \right|^k \leq C^k\, k^{k/2}  \ell^{-\beta k /2} \, \e^{c^* tk^{(2\alpha-\beta)/(\alpha-\beta)} },
\end{equation}where $c^*$ is the constant given in Lemma \ref{lem:moment}
\end{lemma}

\begin{proof}
We may write
\begin{equation}
u_t(x)-u_t^{(\ell)}(x)=\mathcal{E}_1(t,x)+\mathcal{E}_2(x)+\mathcal{E}_3(x),
\end{equation}where
\begin{align*}
\mathcal{E}_1(t,x)&:=\int_{(0,t)\times\R^d} p_{t-s}(y-x)u_s(y)\mathbb{1}_{\mathcal{I}^c_t(x;\ell)}(y) \, F(\d s\, \d y),\\
\mathcal{E}_2(t,x)&:=\int_{(0,t)\times \mathcal{I}_t(x;\ell)} p_{t-s}(y-x)u_s(y) F(\d s\, \d y) - \int_{(0,t)\times \mathcal{I}_t(x;\ell)} p_{t-s}(y-x)u_s(y) F^{(h_\ell)} (\d s\, \d y),\\
\mathcal{E}_3(t,x)&:=\int_{(0,t)\times \mathcal{I}_t(x;\ell)} p_{t-s}(y-x) \left(u_s(y)-u_s^{(\ell)}(y) \right) F^{(h_\ell)}(\d s \, \d y).
\end{align*}
Let us first consider $\mathcal{E}_1$: As in the proof of Lemma \ref{lem:moment} (see \eqref{eq:BDG} and \eqref{eq:BDG1}), and by also \eqref{eq:upper-moment}, we obtain
\begin{align*}
\|\mathcal{E}_1(t,x)\|_k^2 &\leq 4k \int_0^t \sup_{y\in\R^d} \|u_s(y)\|_k^2 \int_{\mathcal{I}^c_t(x;\ell)} \int_{\mathcal{I}^c_t(x;\ell)} p_{t-s}(y-x)p_{t-s}(y'-x) f(y-y') \d y \d y' \d s\\
&\leq 16k\, \e^{c^* t \, k^{\alpha/(\alpha-\beta)}} \int_0^t   \int_{\mathcal{I}^c_t(x;\ell)} \int_{\mathcal{I}^c_t(x;\ell)} p_{t-s}(y-x)p_{t-s}(y'-x) f(y-y') \d y \d y' \d s.
\end{align*}
Since $\int_{\R^d} p_{t-s}(y'-x)f(y-y') \d y' \leq C (t-s)^{-\beta/\alpha}$ for some finite constant $C:=C(\alpha,\beta)>0$ by the Plancherel theorem, we have
\begin{equation}\label{E1}
\begin{aligned}
\|\mathcal{E}_1(t,x)\|_k^2 &\leq 16 C k\, \e^{2c^* t \, k^{\alpha/(\alpha-\beta)}} \int_0^t (t-s)^{-\beta/\alpha} \int_{\mathcal{I}^c_t(x;\ell)} p_{t-s}(y-x) \d y \d s\\
&\leq 16 C k\, \e^{2c^* t \, k^{\alpha/(\alpha-\beta)}}  \int_0^t (t-s)^{1-\beta/\alpha} \int_{\mathcal{I}^c_t(x;\ell)} \|y-x\|^{-(d+\alpha)} \d y \d s\\
&\leq C(t,\alpha,\beta)\, k\, \ell^{-\alpha} \,  \e^{2c^* t \, k^{\alpha/(\alpha-\beta)}} .
\end{aligned}
\end{equation}
We now consider $\mathcal{E}_2$.  Let $\tilde h_\ell(x):=h(x)-h_\ell(x)=c_{\beta,d}\|x\|^{-\beta}\mathbb{1}_{\{\|x\|\geq \ell \}}(x)$ for all $x \in \R^d$.   Then, we have
\begin{align*}
&\int_0^t\int_{\R^d} \left|\int_{\R^d} p_s(y-x) \tilde h_\ell (y-z) \d y \right|^2 \d z \d s\\
&= \int_0^t \int_{\R^d} p_s(y) \int_{\R^d} p_s(y') \int_{\R^d} \tilde h_\ell (y-z) \tilde h_\ell(y'-z)\,  \d z \d y \d y' \d s \\
&\leq \int_0^t \left(\int_{\R^d} p_s(y) \, \d y\right)^2 \int_{\R^d} \tilde h_\ell^2(z) \d z \d s \\
&\leq t\beta^{-1} \ell^{-\beta}, 
\end{align*}
Therefore, by the Burkholder-Davis-Gundy inequality, and then Minkowskii's inequality and Cauchy-Schwartz inequality, we have 
\begin{equation}\label{E2}
\begin{aligned}
\left\|\mathcal{E}_2(t,x)\right\|_k^2 &\leq 4k \left\| \int_0^t \int_{\R^d} \left| \int_{\mathcal{I}_t(x;\ell)} p_{t-s}(y-x) u_s(y) \tilde h_\ell (y-z) \d y \right|^2 \d z \d s  \right\|_{k/2}\\
&\leq 4k \int_0^t \sup_{x\in\R^d} \|u_s(x)\|_k^2\int_{\R^d} \left|\int_{\R^d} p_{t-s}(y-x) \tilde h_\ell (y-z) \d y \right|^2 \d z \d s\\
&\leq 16k \e^{2c^* t \, k^{\alpha/(\alpha-\beta)}}  \int_0^t \int_{\R^d} \left|\int_{\R^d} p_s(y-x) \tilde h_\ell (y-z) \d y \right|^2 \d z \d s\\
&\leq 16k t \beta^{-1} \, \ell^{-\beta}\, \e^{2c^* t \, k^{\alpha/(\alpha-\beta)}}.
\end{aligned}
\end{equation}
Regarding $\mathcal{E}_3$, as in the proof of Lemma \ref{lem:moment} (see \eqref{eq:BDG} and \eqref{eq:BDG1}), we obtain
\begin{equation}\label{E3}
\left\| \mathcal{E}_3(t,x)\right\|_k^2 \leq 2^{(2-\beta)/\alpha}\, \Gamma(\beta/\alpha) k \int_0^t \frac{\sup_{x} \left\|u_s(x)-u_s^{(\ell)}(x)\right\|_k^2}{(t-s)^{\beta/\alpha}} \, \d s.
\end{equation}
Thus, by \eqref{E1}, \eqref{E2} and \eqref{E3}, and $\ell>1, \alpha>\beta$,  we get that, for some constant $C:=C(\alpha,\beta, t)$,
\begin{equation}
\sup_{x\in\R^d} \left\| u_t(x)-u_t^{(\ell)}(x) \right\|_k^2 \leq C\, k\, \e^{c^*\, t\,  k^{\alpha/(\alpha-\beta)}} + 2^{(2-\beta)/\alpha}\, \Gamma(\beta/\alpha) k \int_0^t \frac{\sup_{x} \left\|u_s(x)-u_s^{(\ell)}(x)\right\|_k^2}{(t-s)^{\beta/\alpha}} \, \d s
\end{equation}
Using then the norm $\mathcal{N}_{\gamma,k}$ defined in \eqref{norm} with the same $\gamma:=c^* k^{\alpha/(\alpha-\beta)}$, we obtain
\begin{equation}
\mathcal{N}_{\gamma, k}\left(u-u^{(\ell)} \right) \leq 2C k.
\end{equation}
This completes the proof by relabeling $2C$. 
\end{proof}

By Lemmas \ref{lem:picard} and \ref{lem:u-ul}, and the Chebychev  inequality, we get the following:
\begin{lemma}\label{lem:u-ulm}
Fix $t>0$ and $\lambda\geq 1$. We have
\begin{equation}\label{eq:condLB}
\sup_{x\in\R^d} \P\left\{\left| u_t(x)- u_t^{(\ell,m)}(x)\right| \geq \lambda \right\} \leq C^k\, k^{k/2}\, \left[ \ell^{-\beta k/2}+2^{-mk/2} \right] \exp\left(-k\log\lambda + c^*\, t\,  k^{(2\alpha-\beta)/(\alpha-\beta)}\right). 
\end{equation}
\end{lemma}

We can now prove the lower bound in Theorem \ref{th:pam}. 
\begin{proof}[Proof of the lower bound in Theorem \ref{th:pam}]
We will verify the condition \eqref{cond:LB}. 

Let $X(x):=\log u_t(x)$, $S(x):=\exp(x)$ and $\delta \in (0,1)$. For any $n\geq 1$, we  choose
\begin{align}
\ell:=\ell_n:=\exp\left(n^{(3\alpha-\beta)/(4\alpha-2\beta)} \right)\quad, \quad m:=m_n:=\lfloor{1+\beta\log_2 \ell}\rfloor.
\end{align}
Since $\alpha>\beta>0$ by the assumption \eqref{condition}, for any $x, y \in \Pi_n (\delta)$ with $x\neq y$,
\begin{equation}
 \|x-y\|\geq \exp(\delta n ) \geq  2m_n\ell_n(1+t^{1/\alpha}) \approx \exp\left(n^{(3\alpha-\beta)/(4\alpha-2\beta)} \right)
 \end{equation} for all large $n\geq 1$. This implies that
 $\left\{u_t^{(\ell_n,m_n)}(x_i)\right\}$'s are independent for $\{x_i\}\in \Pi_n(\delta)$ for all large $n\geq 1$ by Lemma \ref{lem:indep:she}. Thus,  we can  define $Y_i:=u_t^{(\ell_n, m_n)}(x_i)$.
 
 We now choose $k:=k_n:=\left[\frac{\beta}{4c^* t} \log \ell_n   \right]^{(\alpha-\beta)/\alpha}$  and plug in $k=k_n, \lambda=1, \ell=\ell_n$ and $m=m_n$ into \eqref{eq:condLB} to obtain  
\begin{equation} 
\log \P\left\{\left| X(x_i)- Y_i\right| \geq \lambda \right\}= \log\P\left\{\left| u_t(x_i)- u_t^{(\ell_n, m_n)}(x_i)\right| \geq \lambda \right\} \leq - C\, n^{\frac{3\alpha-\beta}{2\alpha}} 
\end{equation} where the constant $C$ is independent of $n$. Since $\alpha>\beta$, which implies that $(3\alpha-\beta)/(2\alpha) > 1$, we verify \eqref{cond:LB}. 
\end{proof}

\section*{Acknowledgments} The author would like to thank Professor Davar Khoshnevisan for his valuable comments and continued support. The author would also like to  thank Professor Kyeonghun Kim for his hospitality and support while the author was visiting Korea University.

\begin{spacing}{1}
\begin{small}
\end{small}\end{spacing}
\vskip.1in

\begin{small}
\noindent\textbf{Kunwoo Kim}\\
\noindent Mathematical Sciences Research Institute\\ 17 Gauss Way, Berkeley, CA 94720\\
\noindent \emph{Email address}: \texttt{kunwookim@msri.org}
\end{small}

\end{document}